\documentclass[reqno,final]{tac}
\pdfoutput=1

\usepackage{macros}
\usepackage{hyperref}
\hypersetup{final}

\hypersetup{
  bookmarks=true,
  colorlinks=true,
  linkcolor=black,
  citecolor=black,
  filecolor=black,
  urlcolor=black,
  pdftitle={Finitely Presentable Algebras For Finitary Monads},
  pdfauthor={Ji\v{r}\'\i\ Ad\'amek, Stefan Milius, Lurdes Sousa, and
    Thorsten Wi\ss\/mann},
  pdfkeywords={coalgebras, recursive, well-founded},
  pdfduplex={DuplexFlipLongEdge},
}

\usepackage[notref,notcite]{showkeys}
\usepackage{xcolor}
\usepackage{soul}
\usepackage{seqsplit}
\usepackage{xstring}

\makeatletter
\renewcommand*\showkeyslabelformat[1]{%
\noexpandarg%
\StrSubstitute{#1}{ }{\textvisiblespace}[\TEMP]%
\begin{minipage}[t]{\marginparwidth}%
  \normalfont\small\ttfamily\(\{\){\color{red!50!black}\expandafter\seqsplit\expandafter{\TEMP}}\(\}\)\end{minipage}%
}
\makeatother

\usepackage[marginclue,nomargin,footnote]{fixme}

\newcommand{\takeout}[1]{\empty}
\title{Finitely Presentable Algebras For Finitary Monads}
\author{J.~Ad\'{a}mek, S.~Milius, L.~Sousa and T.~Wi{\ss}mann}
\thanks{J.~Ad\'amek was supported by the Grant Agency of the Czech Republic under the grant 19-009025. \\ S. Milius and T. Wi\ss\/mann acknowledge support by the Deutsche Forschungsgemeinschaft (DFG) under project MI~717/5-2.
\\L.~Sousa was partially supported by the Centre for Mathematics of the
University of Coimbra -- UID/MAT/00324/2019, funded by the Portuguese
Government through FCT/MEC and co-funded by the European Regional Development Fund through the Partnership Agreement PT2020.}
\address{%
  Department of Mathematics, Faculty of Electrical Engineering, Czech Technical University in Prague, Czech Republic\\[5pt]
  Lehrstuhl f\"ur Informatik 8 (Theoretische Informatik), Friedrich-Alexander-Universit\"at Er\-lan\-gen-N\"urn\-berg, Germany \\[5pt]
  CMUC, University of Coimbra,  Portugal \&  ESTGV, Polytechnic Institute of Viseu, Portugal\\[5pt]
}
\eaddress{j.adamek@tu-braunschweig.de\CR mail@stefan-milius.eu\CR
  sousa@estv.ipv.pt\CR thorsten.wissmann@fau.de}

\copyrightyear{2019}

\date{\today}

\begin{document}
%
%
\FXRegisterAuthor{sm}{asm}{SM}
\FXRegisterAuthor{ja}{aja}{JA}
\FXRegisterAuthor{ls}{als}{LS}
\FXRegisterAuthor{tw}{twm}{TW}

\maketitle

\begin{abstract}
  For finitary regular monads $\T$ on locally finitely presentable
  categories we characterize the finitely presentable objects in the
  category of $\T$-algebras in the style known from general algebra:
  they are precisely the algebras presentable by finitely many
  generators and finitely many relations.
\end{abstract}

\section{Introduction}
\label{sec:intro}
If $\T = (T,\eta,\mu)$ is a finitary monad on $\Set$, then the category
$\Set^{\T}$ of its algebras is nothing else than the classical concept of a
variety of algebras. An algebra $A$ is called {\em finitely presentable} (in
General Algebra) if it can be presented by a finite set of generators and a
finite set of relations. This means that
there exists a finite set $X$ (of generators) such that $A$ can be obtained as a
quotient of the free algebra $(TX,\mu_X)$ modulo a finitely generated congruence
$E$. A congruence $E$ is \emph{finitely generated} if there is a finite
subrelation $R \subseteq E$ such that $E$ is the smallest congruence on $TX$
containing $R$.\footnote{Note that this does not imply that $E$, regarded as a
  subalgebra of $(TX, \mu_X)^2$, is a finitely generated algebra.} For monads on
$\Set$, the above concept coincides with $A$ being a finitely presentable object
of $\Set^{\T}$ (see~\cite[Corollary~3.13]{AdamekR94}). In the present paper, we
generalize this to finitary \emph{regular monads}~\cite{Manes76}, i.e.~those
preserving regular epimorphisms, on locally finitely presentable categories $\A$
that have regular factorizations. We introduce the concept of a finitely
generated congruence (see Definition~\ref{D:finGenCong}) and prove that the
finitely presentable objects of $\A^{\T}$ are precisely the quotients of free
algebras $(T X, \mu_X)$ with $X$ finitely presentable modulo finitely generated
congruences. We also characterize
finitely generated algebras for finitary monads; here no condition on the monad
is required.

The presented results can be also formulated for locally
$\lambda$-presentable categories and algebras for $\lambda$-accessible monads
that are $\lambda$-presentable or $\lambda$-generated.

\section{Preliminaries}\label{sec:prel}

In this section we present properties on  finitely presentable and finitely generated objects which will be useful in the subsequent sections.

 Recall that an
object $A$ in a category $\A$ is called \emph{finitely presentable} if its
hom-functor $\A(A,-)$ preserves filtered colimits, and $A$ is called
\emph{finitely generated} if $\A(A,-)$ preserves filtered colimits of
monomorphisms -- more precisely, colimits of filtered diagrams $D\colon \D
\to \A$ for which $Dh$ is a monomorphism in $\A$ for every morphism
$h$ of $\D$.  

\begin{notation}
  For a category $\A$ we denote by
  \[
    \Afp \qquad\text{and}\qquad\Afg
  \]
  full subcategories of $\A$ representing (up to isomorphism) all finitely
  presentable and finitely generated objects, respectively.

  Subobjects $m\colon M \monoto A$ with $M$ finitely generated are called
  \emph{finitely generated subobjects}.
\end{notation}

Recall that $\A$ is a \emph{locally finitely presentable} category, shortly
\emph{lfp} category, if it is cocomplete, $\Afp$ is essentially small, and every
object is a colimit of a filtered diagram in $\Afp$.

We now recall a number of standard facts about
lfp categories~\cite{AdamekR94}.

\begin{remark}\label{R:prelim}
  Let $\A$ be an lfp category.
  \begin{enumerate}
  \item \label{I:factorization} By~\cite[Proposition~1.61]{AdamekR94}, $\A$ has
    (strong epi, mono)-factorizations of morphisms.

  \item \label{I:canColim} By~\cite[Proposition~1.57]{AdamekR94}, every object
    $A$ of $\A$ is the colimit of its \emph{canonical filtered diagram}
    \[
      D_A\colon \Afp/A \to \A \qquad (P \xrightarrow{p} A) \mapsto P, 
    \]
    with colimit injections given by the $p$'s.

  \item \label{I:fingen} By~\cite[Proposition~1.69]{AdamekR94}, an object $A$ is
    finitely generated iff it is a strong quotient of a finitely
    presentable object, i.e.~there exists a finitely presentable
    object $A_0$ and a strong epimorphism $e\colon A_0 \epito A$.
  \item It is easy to verify that every split quotient of a finitely
    presentable object is finitely presentable again.
  \end{enumerate}
\end{remark}

\begin{remark} \label{R:refl}
  Colimits of filtered diagrams $D\colon \D\to \Set$ are precisely those cocones
  $c_i\colon D_i\to C$ ($i\in \obj \D$) of $D$ that have the following
  properties:
  \begin{enumerate}
  \item $(c_i)$ is jointly surjective, i.e.~$C=\bigcup c_i[D_i]$,  and
  \item given $i$ and elements $x,y\in D_i$ merged by $c_i$, then they are also
    merged by a connecting morphism $D_i\to D_j$ of $D$.
  \end{enumerate}
  This is easy to see: for every cocone $c_i'\colon D_i\to C'$ of $D$ define
  $f\colon C\to
  C'$ by choosing for every $x\in C$ some $y\in D_i$ with $x=c_i(y)$ and putting
  $f(x) = c_i'(y)$. By the two properties above, this is well defined and is unique
  with $f\cdot c_i = c_i'$ for all $i$.
\end{remark}

Recall that an adjunction whose right adjoint is finitary is called a
\emph{finitary adjunction}.
\begin{lemma}\label{lem:fppres}
  Let $L \dashv R\colon \B \to \A$ be a finitary adjunction between lfp
  categories. Then we have:
  \begin{enumerate}
  \item $L$ preserves both finitely presentable objects and finitely generated ones;
  \item if $L$ is fully faithful, then an object $X$ is finitely presentable in $\A$ iff  $LX$ is finitely presentable in $\B$;
  \item if, moreover, $L$ preserves monomorphisms, then $X$ is  finitely generated  in $\A$ iff  $LX$ is  finitely generated  in $\B$. 
  \end{enumerate}
\end{lemma}
\proof
  \begin{enumerate}
  \item Let $X$ be a finitely presentable object of $\A$ and let $D\colon \D \to \B$ be a
    filtered diagram. Then we have the following chain of natural
    isomorphisms
    \begin{align*}
      \B(LX, \colim D) 
      & \cong \A(X, R(\colim D))\\
      & \cong \A(X, \colim RD) \\
      & \cong \colim(\A(X, RD(-)) \\
      & \cong \colim(\B(LX, D(-)).
    \end{align*}
    This shows that $LX$ is finitely presentable in $\B$. Now if $X$ is  finitely generated  in $\A$ and
    $D$ is a directed diagram of monomorphisms, then $RD$ is also a directed diagram of
    monomorphisms (since the right adjoint $R$ preserves monomorphisms). Thus, the
    same reasoning proves $LX$ to be  finitely generated  in $\B$. 
  \item Suppose that $LX$ is finitely presentable in $\B$ and that $D\colon \D \to \A$ is a
    filtered diagram. Then we have the following chain of natural
    isomorphisms:
    \begin{align*}
      \A(X,\colim D) & \cong \B(LX, L(\colim D)) \\
      & \cong \B(LX, \colim LD )\\
      & \cong \colim (\B(LX, LD(-)) \\
      & \cong \colim (\A(X,D(-))
    \end{align*}
    Indeed, the first and last step use that $L$ is fully faithful,
    the second step that $L$ is finitary and the third one that
    $LX$ is finitely presentable in $\B$. 
  \item If $LX$ is finitely generated in $\B$ and $D\colon \D \to \A$ a directed
    diagram of monomorphisms, then so is $LD$ since $L$ preserves monomorphisms.
    Thus the same reasoning as in~(2) shows that $X$ is finitely generated in
    $\A$. \endproof
  \end{enumerate}
\begin{lemma} \label{L:power}
  Let $\A$ be an lfp category and $I$ a set. An object in the power category $\A^I$ is
  finitely presentable iff its components
  \begin{enumerate}
  \item\label{L:power:allfp} are finitely presentable in $\A$, and
  \item\label{L:power:most0} all but finitely many are initial objects.
  \end{enumerate}
\end{lemma}

\begin{proof}
  Denote by $0$ and $1$ the initial and terminal objects, respectively. Note that for every $i\in I$ there are two fully faithful functors
  $L_i,R_i\colon \A\hookrightarrow \A^I$ defined by:
  \[
    \big(L_i(X)\big)_j = \begin{cases}
      X&\text{if }i = j\\
      0&\text{if }i \neq j\\
    \end{cases}
    \qquad
    \text{ and }
    \qquad
    \big(R_i(X)\big)_j = \begin{cases}
      X&\text{if }i = j\\
      1&\text{if }i \neq j\\
    \end{cases}
  \]
  For every $i\in I$ there is also a canonical projection $\pi_i\colon \A^I\to
  \A$, $\pi_i\big((X_j)_{j\in I}\big) = X_i$.  We have the following adjunctions:
  \[
    L_i \dashv \pi_i \dashv R_i.
  \]
  
  Sufficiency.  Let
  $A = (A_i)_{i\in I}$ satisfy \ref{L:power:allfp} and
  \ref{L:power:most0}, then $L_i(A_i)$ is finitely presentable in $\A^I$
  by Lemma~\ref{lem:fppres}(1). Thus, so is $A$, since it is the
  finite coproduct of those $L_i(A_i)$, with $A_i$ not
  initial. Obviously, $L_i(A_i)$ is finitely presentable.

  Necessity. Let $A =(A_i)_{i \in I}$ be finitely presentable in
  $\A^I$. Then for every $i\in I$, $\pi_i(A)$ is finitely presentable
  in $\A$ by Lemma~\ref{lem:fppres}(1), proving item
  \ref{L:power:allfp}. To verify \ref{L:power:most0}, for every finite
  set $J\subseteq I$, let $A_J$ have the components $A_j$ for every
  $j\in J$ and $0$ otherwise. These objects $A_J$ form an obvious
  directed diagram with a colimit cocone $a_J\colon A_J\to A$. Since
  $A$ is finitely presentable, there exists $J_0$ such that $\id_A$
  factorizes through $a_{J_0}$, i.e.~$a_{J_0}$ is a split
  epimorphism. Since a split quotient of an initial object is initial,
  we conclude that~\ref{L:power:most0} holds.
\end{proof}

\section{Finitely Presentable Algebras}\label{sec:alg}
\label{sec:genpres}

In the introduction we have recalled the definition of a finitely
presentable algebra from General Algebra and the fact that for a
finitary monad $\T$ on $\Set$, this is equivalent to $A$ being a
finitely presentable object of $\Set^\T$. We now generalize this to finitary
\emph{regular monads}~\cite{Manes76}, i.e.~those preserving
regular epimorphisms, on lfp categories that have regular
factorizations.

First, we turn to characterizing finitely generated algebras for
\emph{arbitrary} finitary monads.

\begin{remark}
  Let $\T$ be a finitary monad on an lfp category $\A$. Then the
  Eilenberg-Moore category $\A^\T$ is also
  lfp~\cite[Remark~2.78]{AdamekR94}. Thus, it has (strong epi,
  mono)-factorizations. The monomorphisms in $\A^\T$, representing
  \emph{subalgebras}, are precisely the $\T$-algebra morphisms
  carried by a monomorphism of $\A$ (since the forgetful functor $\A^\T
  \to \A$ creates limits). The strong epimorphisms of $\A^\T$,
  representing \emph{strong quotient algebras}, need not coincide with
  those carried by strong epimorphisms of $\A$ -- we do not assume that $\T$ preserves strong epimorphisms.
\end{remark}
Recall our terminology that a finitely generated subobject of an object $A$  is
one represented by a monomorphism $m\colon  M \monoto A$  with $M$ a finitely generated
object. 
\begin{notation} Throughout this section given a $\T$-algebra morphism $f\colon
  X\to Y$ we denote by $\Im f$ its image in $\A^{\T}$. That is, we have a strong
  epimorphism $e\colon X\epito \Im f$ and a monomorphism $m\colon \Im f\monoto B$ in $\A^\T$ with $f=m\cdot e$.
\end{notation}
\begin{definition}
  An algebra $(A,a)$ for $\T$ is said to be \emph{generated by
    a subobject} $m\colon  M \monoto A$ of the base category $\A$ if no
  proper subalgebra of $(A,a)$ contains $m$.  
\end{definition}
The phrase ``$(A,a)$ is generated by a finitely generated subobject'' may
sound strange, but its meaning is clear: there exists a subobject $m\colon 
M \monoto A$ with $M$ in $\Afg$ such that $m$ does not factorize
through any proper subalgebra of $(A,a)$.
\begin{example}\label{E:ffp}
  The free algebras on finitely presentable objects are shortly called
  \emph{ffp algebras} below: they are the algebras $(TX, \mu_X)$ with $X$
  finitely presentable. 
  \begin{enumerate}
  \item Every ffp algebra is generated by a finitely
  generated object: factorize the unit $\eta_X\colon  X \to TX$ in $\A$ as a strong
  epimorphism $e\colon  X \epito M$ (thus, $M$ is finitely generated by
  Remark~\ref{R:prelim}(5)) followed by a monomorphism $m\colon  M \monoto
  TX$. Using the universal property, it is easy to see that $m$
  generates $(TX, \mu_X)$; indeed, suppose we had a subalgebra $s\colon 
  (A,a) \monoto (TX, \mu_X)$ containing $m$, via $n\colon  M \monoto A$,
  say. Then the unique extension of $n \cdot e\colon  X \to A$ to a
  $\T$-algebra morphism $h\colon  (TX, \mu_X) \to (A,a)$ satisfies $s \cdot
  h = \id_{(TX,\mu_X)}$. Thus, $s$ is an isomorphism. 
\item \label{ffpIsFp} Every ffp algebra is finitely presentable in $\A^{\T}$: apply
  Lemma~\ref{lem:fppres} 
  to the forgetful functor
  $R\colon\A^{\T}\to \A$ and its left adjoint $LX=(TX, \mu_X)$.
  \end{enumerate}
\end{example}
\begin{theorem}\label{T:fg}
  For every finitary monad $\T$ on an lfp category $\A$ the following
  conditions on an algebra $(A,a)$ are equivalent:
  \begin{enumerate}
  \item $(A,a)$ is generated by a finitely generated subobject,
  \item $(A,a)$ is a strong quotient algebra of an ffp algebra, and
  \item $(A,a)$ is a finitely generated object of $\A^\T$.
  \end{enumerate}
\end{theorem}
\begin{proof}
  (3)$\implies$(2) First observe that the cocone
  $Tf\colon  TX \to TA$, where $(X,f)$ ranges over $\Afp/A$, is collectively
  epimorphic since $T$ preserves the filtered colimit
  $A = \colim D_A$ of Remark \ref{R:prelim}(2). For every $f\colon  X \to A$ consider its unique
  extension to a $\T$-algebra morphism
  $a \cdot Tf\colon  (TX, \mu_X) \to (A,a)$ and form its
 factorization in $\A^{\T}$:
  \[
    \xymatrix{
      TX
      \ar[d]_{Tf}
      \ar@{->>}[r]^-{e_f}
      &
      \Im(a \cdot Tf)
      \ar@{ >->}[d]^{m_f}
      \\
      TA
      \ar[r]_-{a}
      &
      A        
    }
  \]
  Now observe that $a\colon  (TA,\mu_A) \to (A,a)$ is a strong epimorphism in $\A^\T$; in fact,
  the laws of Eilenberg-Moore algebras for $\T$ imply that
  $a$ is the coequalizer of
  \[
    \xymatrix@1{
      (TTA, \mu_{TA})
      \ar@<3pt>[r]^-{Ta}
      \ar@<-3pt>[r]_-{\mu_A}
      &
      (TA, \mu_A)}.
  \]
  From Remark \ref{R:prelim}\ref{I:canColim} and the finitarity of the
  functor $T$ we deduce that $Tf\colon (TX,\mu_X)\to (TA,\mu_A)$,
  $f\in \A_\fp/A$, is a filtered colimit in $\A^{\T}$.  It follows
  from Lemma~\cite[Lemma 2.9]{amsw19}\twnote{This is about Lemma \{L:im\}}
  that $(A,a)$ is a directed union of its
  subobjects $m_f$ for $f$ in $\Afp/A$.
  
  Now since $(A,a)$ is finitely generated, $\id_A$ factorizes through
  one of the corresponding colimit injections $m_f\colon  \Im(a \cdot Tf)
  \monoto A$ for some $f\colon X\to A$ in $\Afp/A$. Therefore $m_f$ is  split
  epic, whence an isomorphism, and $A$ is a strong quotient of $(TX,
  \mu_X)$ via $e_f$, as desired. 
  
  \medskip\noindent
  (2)$\implies$(1)
  Let $q\colon (TX, \mu_X) \epito (A,a)$ be a strong epimorphism in $\A^\T$
  with $X$ finitely presentable in $\A$. Factorize $q \cdot \eta_X$ as a
  strong epimorphism followed by a monomorphism in $\A$:
  \[
    \xymatrix{
      X \ar[r]^-{\eta_X} \ar@{->>}[d]_e & TX \ar@{->>}[d]^-q \\
      M \ar@{ >->}[r]_-m & A
    }
  \]
  Then $M$ is finitely generated in $\A$ by Remark~\ref{R:prelim}(5). We shall
  prove that for every subalgebra $u\colon (B,b) \monoto (A,a)$ containing $m$ (i.e.\
  such that there is a morphism $g\colon M \to B$ in $\A$ with $u \cdot g =
  m$) $u$ is invertible. Let $\ext e\colon (TX, \mu_X) \to (B,b)$ be the
  unique extension of $g \cdot e$ to a $\T$-algebra morphism:
  \[
    \xymatrix{
      X
      \ar@{->>}[dd]_e
      \ar[rr]^-{\eta_X}
      &&
      TX
      \ar@{->>}[dd]^-q
      \ar[ld]_{\ext e}
      \\
      &
      B
      \ar[rd]^-u
      \\
      M \ar@{ >->}[rr]_-m \ar[ru]^-{g} && A
    }
  \]
  Then we see that $u \cdot \ext e = q$ because this triangle of
  $\T$-algebra morphisms commutes when precomposed by the universal
  morphism $\eta_X$. Since $q$ is strongly epic, so is $u$, and
  therefore $u$ is an isomorphism, as desired. 
  
  \medskip\noindent (1)$\implies$(2) Let $m\colon M \monoto A$ be a
  finitely generated subobject of $A$ that generates $(A,a)$. By
  Remark~\ref{R:prelim}(5), there exists a strong epimorphism
  $q\colon X \epito M$ in $\A$ with $X$ finitely presentable. The unique
  extension $e = \ext{(m \cdot q)}\colon (TX, \mu_X) \to (A,a)$ to a
  $\T$-algebra morphism is an \emph{extremal epimorphism}, i.e.~if
  $e$ factorizes through a subalgebra $u\colon (B,b) \monoto (A,a)$, then
  $u$ is an isomorphism. To prove this, recall that $u$ is also monic
  in $\A$. Given $e = u \cdot e'$ we use the diagonal fill-in property
  in $\A$:
  \[
    \xymatrix{
      X
      \ar@{->>}[r]^-q
      \ar[d]_{\eta_X}
      &
      M
      \ar@{ >->}[dd]^m
      \ar@{-->}[ldd]
      \\
      TX
      \ar[d]_{e'}
      \\
      B
      \ar@{ >->}[r]_- u
      &
      A
    }
  \]
  Since $m$ generates $(A,a)$, this proves that $u$ is an
  isomorphism. In a complete category every extremal epimorphism
  is strong, thus we have proven (2).

  By Remark~\ref{R:prelim}(5) and the fact that ffp algebras are finitely presentable in $\A^\T$ (see Example \ref{E:ffp}(b)) we have (2)$\implies$(3).
\end{proof}

As usual, by a \emph{congruence} on a $\T$-algebra $(A,a)$ is meant a
subalgebra $(K,k) \monoto (A,a) \times (A,a)$ forming a kernel pair $\ell,r\colon (K,k)
\parallel (A,a)$ of some $\T$-algebra morphism. Given a coequalizer $q\colon (A,a)
\epito (B,b)$  of $\ell, r$ in $\A^\T$, then $(B,b)$ is called the
\emph{quotient algebra of $(A,a)$ modulo $(K,k)$}. 
\begin{definition} \label{D:finGenCong}
  By a \emph{finitely generated congruence} is meant
  a congruence $\ell, r\colon (K,k) \parallel (A,a)$ such that there exists a finitely
  generated subalgebra $m\colon (K',k') \monoto (K,k)$ in $\A^\T$ for which
  the quotient of $(A,a)$ modulo $(K,k)$ is also a coequalizer of $\ell \cdot
  m$ and $r \cdot m$:
  \[
    \xymatrix@1{
      (K',k') \ar@{ >->}[r]^-m
      &
      (K,k) \ar@<3pt>[r]^-\ell \ar@<-3pt>[r]_-r
      &
      (A,a)
      \ar@{->>}[r]^-q
      &
      (B,b).
    }
  \]
\end{definition}
In the next theorem we assume that our base category has regular factorizations,
i.e.~(regular epi,mono)-factorizations.
\begin{theorem}
  \label{T:coequffp}
  Let $\T$ be a regular, finitary monad  on an
  lfp category $\A$ with regular factorizations. For every
  $\T$-algebra $(A,a)$ the following conditions are equivalent:
  \begin{enumerate}
  \item $(A,a)$ is a quotient of an ffp algebra modulo a
    finitely generated congruence,
    
  \item $(A,a)$ is a coequalizer of a parallel pair of $\T$-algebra
    morphisms between ffp algebras:
    \[
      \xymatrix@1{
        (TY, \mu_Y) \ar@<3pt>[r]^-f \ar@<-3pt>[r]_-g & (TX, \mu_X)
        \ar@{->>}[r]^-e & (A,a)
      }
      \qquad
      (X,Y \text{ in }\A_\fp),
      \qquad
      \text{and}
    \]
  \item $(A,a)$ is a finitely presentable object of $\A^\T$. 
  \end{enumerate}
\end{theorem}
\begin{proof}
  (2)$\implies$(3) Since finitely presentable objects are closed under
  finite colimits, this follows from Example \ref{E:ffp}\ref{ffpIsFp}. 

  \medskip\noindent (3)$\implies$(1) First note that the classes of
  regular and strong epimorphisms in $\A^\T$ coincide; indeed, since
  $\T$ preserves regular epimorphisms, the regular factorizations of
  $\A$ lift to $\A^\T$ (see \cite[Proposition~4.17]{Manes76}). Then, by
  Theorem~\ref{T:fg}, $(A,a)$ is a regular quotient of an ffp algebra
  via $q\colon (TX, \mu_X) \epito (A,a)$, say.

  Now take the kernel pair $\ell,r\colon (K,k) \parallel (TX, \mu_X)$ of $q$ in
  $\A^\T$ and note that $q$ is its coequalizer. We are going to prove that the
  congruence $(\ell,r)$ is finitely generated. We first verify some of its
  properties.
  Write $K$ in $\A$ as the
  filtered colimit of its canonical filtered diagram $D_K\colon \Afp/K \to
  \A$ (see Remark~\ref{R:prelim}(2)) with the
  colimit injections $y\colon  Y \to K$. Take the unique extension $\ext y\colon
  (TY,
  \mu_Y) \to (K,k)$ to a $\T$-algebra morphism and form
  the following coequalizer in $\A^{\T}$:
  \[
    \xymatrix@1{
      (TY, \mu_Y) \ar[r]^-{\ext y}
      &
      (K,k) \ar@<3pt>[r]^-\ell \ar@<-3pt>[r]_-r
      &
      (TX, \mu_X)\ar@{->>}[r]^-{e_y}
      &
      (A_y, a_y)
    }.
  \]
  
  (a)~This defines a filtered diagram $\ol D\colon  \Afp/K \to \A^\T$ taking $y$ to $(A_y,a_y)$. In fact,
  for every morphism $f\colon  (Y,y) \to (Z,z)$ in $\Afp/K$ we obtain 
  a $\T$-algebra morphism $a_f\colon  (A_y, a_y) \to (A_z, a_z)$ using the
  following diagram in $\A^\T$ (where we drop the algebra structures):
  \[
    \xymatrix@R-1pc{
      TY \ar[dd]_{Tf} \ar[rd]^-{\ext y}
      &&&
      A_y \ar@{-->}[dd]^{a_f}
      \\
      &
      K \ar@<3pt>[r]^-\ell \ar@<-3pt>[r]_-r
      &
      TX \ar@{->>}[ru]^-{e_y} \ar@{->>}[rd]_-{e_z}
      \\
      TZ \ar[ru]_-{\ext z}
      &&&
      A_z
    }
  \]
  Note that $a_f$ is a regular epimorphism in $\A^\T$. 
  Furthermore, for every $y\colon  Y \to K$ in $\Afp/K$ we also obtain a
  morphism $d_y\colon  A_y \to A$ such that $d_y \cdot e_y = q$:
  \[
    \xymatrix{
      TY \ar[r]^-{\ext y}
      &
      K
      \ar@<3pt>[r]^-\ell \ar@<-3pt>[r]_-r
      &
      TX
      \ar@{->>}[d]_q
      \ar@{->>}[r]^-{e_y}
      &
      A_y
      \ar@{-->}[ld]^-{d_y}
      \\
      && A
    }
  \]
  These morphisms $d_y$ form a cocone on the diagram $\ol D$; indeed, we
  have for every morphism $f\colon  (Y,y) \to (Z,z)$ of $\Afp/K$ that
  \[
    d_z \cdot a_f \cdot e_y = d_z \cdot e_z = q = d_y \cdot e_y,
  \]
  and we conclude that $d_z \cdot a_f = d_y$ since $e_y$ is epic.

  (b)~We now show that $(A,a) = \colim \ol D$ with colimit injections $d_y:
  (A_y,a_y) \to (A,a)$. Given a cocone $b_y\colon  (A_y,a_y) \to (B,b)$ of $\ol D$,
  we prove that it factorizes uniquely through $(d_y)$. We first note that all
  the morphisms $b_y \cdot e_y$ are equal because the diagram is filtered and
  for every morphism $f\colon  (Y,y) \to (Z,z)$ in $\Afp/K$ we have the commutative
  diagram below:
  \[
    \xymatrix{
      & A_y \ar[rd]^-{b_y} \ar[dd]^{a_f}\\
      TX\ar@{->>}[ru]^-{e_y} \ar@{->>}[rd]_-{e_z} && B \\
      & A_z \ar[ru]_-{b_z}
    }
  \]
  Let us call the above morphism $q'\colon  TX \to B$, and observe that for every
  $y\colon  Y \to K$ in $\Afp/K$ we have
  \[
    q' \cdot \ell \cdot \ext y = b_y \cdot e_y \cdot \ell \cdot \ext y
    = b_y \cdot e_y \cdot r \cdot \ext y = q' \cdot r \cdot \ext y.
  \]
  The cocone of morphisms $\ext y\colon  TY \to K$ is collectively epic since so is
  the colimit cocone $y\colon  Y \to K$, and therefore
  $q' \cdot \ell = q' \cdot r$. Thus, there exists a unique
  factorization $h\colon  A \to B$ of $q'$ through $q = \mathsf{coeq}(\ell,r)$,
  i.e.~$h \cdot q = q'$. We now have, for every
  $y\colon  Y \to K$ in $\Afp/K$,
  \[
    h \cdot d_y \cdot e_y = h \cdot q = q' = b_y \cdot e_y,
  \]
  which implies $h\cdot d_y = b_y$ using that $e_y$ is epic.

  Uniqueness of $h$ with the latter property follows immediately from $q$ being
  epic: if $k\colon  A \to B$ fulfils $k \cdot d_y = b_y$ for every $y$ in $\Afp/K$,
  we have
  \[
    k\cdot q = k \cdot d_y \cdot e_y = b_y \cdot e_y = q'=h\cdot q.
  \]
    
  (c)~Now use that $(A,a)$ is finitely presentable in $\A^\T$ to
  see that there exists some $w\colon  W \to K$ in $\Afp/K$ and a
  $\T$-algebra morphism $s\colon  (A,a) \to
  (A_w,a_w)$ such that $d_w \cdot s = \id_A$. Then $s \cdot d_w$ is an
  endomorphism of the $\T$-algebra $(A_w, a_w)$ satisfying $d_w \cdot
  (s\cdot d_w) = d_w$. Since $e_w$ is a coequalizer of a parallel pair of
  $\T$-algebra morphisms between ffp algebras, $(A_w,a_w)$ is finitely
  presentable by Example \ref{E:ffp}\ref{ffpIsFp}. The colimit injection $d_w$ merges $s
  \cdot d_w$ and $\id_{A_w}$, hence there exists a morphism $f\colon  (W,w) \to
  (Y,y)$ of $\Afp/K$ with $a_f$ merging them too, i.e.~such that $a_f \cdot (s \cdot d_w) = a_f$. This
  implies that $d_y\colon  (A_y, a_y) \to (A,a)$ is an isomorphism with
  inverse $a_f \cdot s$. Indeed, we have
  \[
    d_y \cdot (a_f \cdot s) = d_w \cdot s = \id_A, 
  \]
  and for $(a_f \cdot s) \cdot d_y = \id_{A_y}$ we use that $a_f$ is
  epic:
  \[
    (a_f \cdot s) \cdot d_y \cdot a_f = a_f \cdot s \cdot d_w = a_f.
  \]                                    
  Thus $(A,a)$ is a quotient of the ffp algebra $(TX,\mu_X)$ modulo the
  congruence $(\ell,r)$.

  (d)~We are ready to prove that $\ell, r\colon  (K,k)\parallel (TX,\mu_X)$ is a 
  finitely generated congruence. Take the regular factorization of
  $\ext y\colon  (TY, \mu_Y) \to (K,k)$ for $y$ in (c):
  \[
    \ext y = \big(\xymatrix{
      (TY, \mu_Y)\ar@{->>}[r]^-e
      & ~\Im(y^\sharp)~ \ar@{>->}[r]^-m
      & (K,k)}\big).
  \]
  Then $e_y$ is also the coequalizer of $\ell \cdot m$ and
  $r \cdot m$, and $\Im(y^\sharp)$ is a finitely generated $\T$-algebra by
  Theorem~\ref{T:fg}, as desired.
  
  \medskip\noindent (1)$\implies$(2) We are given a regular
  epimorphism $e\colon  (TX, \mu_X) \epito (A,a)$ with $X$ finitely
  presentable in $\A$ and a pair
  $\ell', r'\colon  (K',k') \parallel (TX, \mu_X)$ with $(K',k')$ finitely
  generated, whose coequalizer is $e$. By
  Theorem~\ref{T:fg}, there exists a regular quotient
  $q\colon  (TY, \mu_Y) \epito (K',k')$ with $Y$ finitely presentable in
  $\A$. Since $e$ is a coequalizer of $\ell', r'$, it is also a
  coequalizer of the pair $\ell' \cdot q, r' \cdot q$.
\end{proof}
\begin{openproblem}
  Are (1)--(3) above equivalent for all finitary monads (not necessarily regular
  ones)?
\end{openproblem}

\section{Finitary Monads on Sets}

We have seen in~\cite[Corollary~3.31]{amsw19}
that finitely presentable objects of $[\Set,\Set]_\fin$ are precisely
the finitely generated ones.  In contrast, we will show that in the
category $\Mndf(\Set)$ of finitary monads on $\Set$ the classes of
finitely presentable and finitely generated objects do \emph{not}
coincide.

\begin{remark}
  The finitely generated objects of the category $[\Set,\Set]_\fin$ of all
  finitary set functors were characterized in \cite{amsw19} as precisely the
  \emph{super-finitary} set functors. These are the quotient functors of the
  polynomial functors $H_\Sigma X = \coprod_{n\in \Nat} {\Sigma_n}\times X^n$ for
  signatures $\Sigma$ of finitely many finitary operations.
\end{remark}

\takeout{
\begin{remark}
  We apply Lack's result that the category of finitary monads on an lfp category
  $\A$ is monadic over the category $\op{Sig}(\A)$ of signatures (which we
  recall in Example \ref{E:ff}), see Corollary 3 of \cite{lack99}. By using the
  same proof one can see that the category of finitary endofunctors is also
  monadic over $\op{Sig}(\A)$.
\end{remark}
}

\begin{example} \label{E:ff} As an application of Theorem \ref{T:fg}, we
  generalize from~\cite[Corollary~3.31]{amsw19}
  the fact
  that $[\Set,\Set]_\fin$ has as finitely generated objects precisely the
  super-finitary functors, to all lfp
  categories $\A$. Denote by
  \[
    [\A,\A]_\fin
  \]
  the category of all finitary endofunctors of $\A$. An example is the
  polynomial functor $H_\Sigma$ for every signature $\Sigma$ in the sense
  of Kelly and Power~\cite{kellypower93}. This means a collection of objects $\Sigma_n$ of
  $\A$ indexed by objects $n\in \A_\fp$. Let $|\A_\fp|$ be the discrete category of
  objects of $\A_\fp$, then the functor category
  \[
    \op{Sig}(\A) = \A^{|\A_\fp|}
  \]
  is called the \emph{category of signatures}. Its morphisms from $\Sigma \to
  \Sigma'$ are collections of morphisms $e_n\colon \Sigma_{n}\to \Sigma_n'$ for
  $n\in|\A_\fp|$. The {\em polynomial functor} $H_\Sigma$ is the coproduct of the
  endofunctors $\A(n,-)\bigcdot \Sigma_n$, where $\bigcdot$ denotes copowers of
  $\Sigma_n$, shortly:
  \[
    H_\Sigma X=\coprod_{n\in \A_\fp}
    \A(n,X) \bigcdot \Sigma_n.
  \]
  We obtain an adjoint situation
  \[
    \xymatrix@1{
      [\A,\A]_\fin~
      \ar@<1ex>[r]^-{U}_{\top}
      &
      ~\op{Sig}(\A)
      \ar@<1ex>[l]^-{\Phi}
    }
  \]
  where the forgetful functor $U$ takes a finitary endofunctor $F$ to the
  signature
  \[
    UF = \Sigma\quad\text{with}\quad
    \Sigma_n = Fn\qquad (n\in \A_{\fp})
  \]
  and $\Phi$ takes a signature $\Sigma$ to the polynomial endofunctor
  $\Phi\Sigma = H_\Sigma$.
  The resulting monad $\mathbb{T}$ is given by
  \[
    (T\Sigma)_n = \coprod_{m\in \A_\fp} \A(m,n)\bigcdot \Sigma_m.
  \]
\end{example}
\begin{remark}
  It is easy to see that the forgetful functor $U$ is monadic. Indeed,
  this follows from Beck's theorem since $U$ has a left-adjoint and
  creates all colimits. The latter is clear since colimits of
  functors are formed object-wise, and for finitary functors they are
  determined on the finitely presentable objects.

  Thus, we see that the category $[\A,\A]_\fin$ is equivalent to
  the Eilenberg-Moore category of the monad $\mathbb{T}$. By Theorem \ref{T:fg},
  finitely generated objects of $[\A,\A]_\fin$ are precisely the strong
  quotients of ffp algebras for $\mathbb{T}$. Now by
  Lemma~\ref{L:power},
  a signature $\Sigma$ is finitely presentable in $\A^{|\A_\fp|}$ iff for the
  initial object $0$ of $\A$ we have
  \[
    \Sigma_n = 0\text{ for all but finitely many }n\in \A_\fp
  \]
  and
  \[
    \Sigma_n\text{ is finitely presentable for every }n\in \A_\fp.
  \]
  Let us call such signatures \emph{super-finitary}. We thus obtain the
  following result.
\end{remark}
\begin{corollary} \label{C:ff}
  For an lfp category $\A$, a finitary endofunctor is finitely generated in
  $[\A,\A]_\fin$ iff it is a strong quotient of a polynomial functor $H_\Sigma$
  with $\Sigma$ super-finitary.
\end{corollary}
\begin{example}
  Another application of results of Section~\ref{sec:alg} is to the category
  \[
    \Mndf(\A)
  \]
  of all finitary monads on an lfp category $\A$. Lack proved in
  \cite{lack99} that this category is also monadic over the category of signatures.
  More precisely, for the forgetful functor $V\colon \Mndf(\A)\to [\A,\A]_\fin$
  the composite
  \[
    UV\colon \Mndf(\A) \to \op{Sig}(\A)
  \]
  is a monadic functor. Recall from Barr \cite{Barr} that every finitary
  endofunctor $H$ generates a free monad; let us denote it by $H^*$. The
  corresponding free monad $\mathbb{T}$ for $UV$ assigns to every signature
  $\Sigma$ the signature derived from the free monad on $\Sigma$ (w.r.t.~$UV$),
  or, equivalently, from the free monad $H_\Sigma^*$ on the polynomial
  endofunctor $H_\Sigma$. Thus the monad $\mathbb{T}$ is given by the following
  rule for $\Sigma$ in $\op{Sig}(\A)$:
  \[
    (T\Sigma)_n = H^*_\Sigma n
    \qquad\text{for all }n\in \A_\fp.
  \] 
  (Example: if $\A=\Set$ then $H_\Sigma^*$ assigns to
  every set $X$ the set $H_\Sigma^*X$ of all $\Sigma$-terms with variables in $X$.)
  In general, it follows from \cite{adamek1974} that the underlying functor of
  $H^*_\Sigma$ is the colimit of the following $\omega$-chain in $[\A,\A]_\fin$:
  \[
    \Id
    \xrightarrow{~w_0~}
    H_\Sigma+\Id
    \xrightarrow{~H_\Sigma w_0 + \id~}
    H_\Sigma(H_\Sigma+\Id)+\Id
    \longrightarrow \cdots
    W_n\xrightarrow{w_n}
    W_{n+1}
    \longrightarrow \cdots
  \]
  Here, $W_0=\Id$ and $W_{n+1} = H_\Sigma W_n+\Id$. Moreover, $w_0\colon \Id\to
  H_\Sigma+\Id$ is the coproduct injection, while $w_{n+1} = H_\Sigma w_n +\id$.
  The monad $H^*_\Sigma$ is thus the free $\mathbb{T}$-algebra on $\Sigma$ and
  the ffp algebras are precisely $H^*_\Sigma$ for $\Sigma$ super-finitary.
\end{example}
\begin{definition}
  Let $\Sigma$ be a signature in an lfp category $\A$.
  \begin{enumerate}
    \item By a \emph{$\Sigma$-equation} is meant a parallel pair
      \[
        f,f'\colon k\longrightarrow H^*_\Sigma n
        \qquad\text{ with }k,n\in \A_\fp
      \]
      of morphisms in $\A$.
    \item A quotient $c\colon H^*_\Sigma \to M$ in $\Mndf(\A)$ is said to \emph{satisfy}
      the equation if its $n$-component merges $f$ and $f'$ (i.e.~$c_n\cdot f =
      c_n\cdot f'$).
    \item By a \emph{presentation} of a monad $\M$ in $\Mndf(\A)$ is meant a
      signature $\Sigma$ and a collection of $\Sigma$-equations such that 
      the least quotient of $H^*_{\Sigma}$ satisfying all of the given equations has the form $c\colon H^*_\Sigma\twoheadrightarrow \M$.
  \end{enumerate}
\end{definition}

\vspace{1mm}

If $\A=\Set$, this is the classical concept of a presentation of a variety by a
signature and equations.
Indeed, given a pair $f,f'\colon 1\to H^*_\Sigma n$, which is a pair of
$\Sigma$-terms in $n$ variables, satisfaction of the equation $f=f'$ in the
sense of General Algebra means precisely $c_n\cdot f = c_n\cdot f'$. And a
general pair $f,f'\colon k\to H^*_\Sigma n$ can be substituted by $k$ pairs of
terms in $n$ variables.
\begin{remark} \label{R:pres}
  \begin{enumerate}
  \item Every finitary monad $\M$ has a presentation. Indeed, since this is an  algebra for
    the monad $\mathbb{T}$ of Example~\ref{E:ff}, it is a coequalizer of a parallel pair of monad
    morphisms between free algebras for $\mathbb{T}$:
    \[
      \xymatrix@1{
      H^*_\Gamma
      \ar@<4pt>[r]^-{\ell}
      \ar@<-1pt>[r]_-{r}
      & H^*_\Sigma
      \ar[r]^-{c}
      & \M
      }
    \]
    To give a monad morphism $\ell$ is equivalent to giving
    a signature morphism
    \[
      \ell_n\colon \Gamma_n\longrightarrow H^*_\Sigma n
      \qquad(n\in |\A_\fp|).
    \]
    Analogously for $r\mapsto (r_n)$. Thus, to say that $c$ merges $\ell$ and
    $r$ is the same as to say that $\M$ satisfies the equations
    $\ell_n,r_n\colon \Gamma_n\to H^*_\Sigma n$  for all $n$. And the above
    coequalizer  $c$ is the least such quotient. 
  \item Every equation $f,f'\colon k\to H^*_\Sigma n = \colim_{r\in \Nat} W_r n$
    can be substituted, for some number $r$ (the ``depth'' of the terms), by an
    equation $g,g'\colon k\to W_r n$. This follows from $k$ being finitely presentable.
  \end{enumerate}
\end{remark}
\begin{theorem}
  Let $\A$ be an lfp category with regular factorizations. A finitary monad is,
  as an object of $\Mndf(\A)$,
  \begin{enumerate}
  \item finitely generated iff it has a presentation by $\Sigma$-equations for a
    super-finitary signature~$\Sigma$,
 and
  \item finitely presentable iff it has a presentation by finitely many
    $\Sigma$-equations for a super-finitary signature $\Sigma$ .
  \end{enumerate}
\end{theorem}
\begin{proof}
  (1)~Let $\M$ have a presentation with $\Sigma$ super-finitary. Then $\M$ is a
  (regular) quotient of an ffp-algebra $H^*_\Sigma$ for $\mathbb{T}$, thus, it
  is finitely generated by Theorem~\ref{T:fg}.

  Conversely, if $\M$ is finitely generated, it is a (strong) quotient $c\colon
  H^*_\Sigma\twoheadrightarrow \M$ for $\Sigma$ super-finitary. It is sufficient
  to show that $c$ is a regular epimorphism in $\Mndf(\A)$, then the argument
  that $\M$ has a presentation using $\Sigma$ is as in Remark~\ref{R:pres}.

  Since $\A$ has regular factorizations, so does $\op{Sig}(\A) = \A^{|\A_{\fp}|}$.
  And the monad $\mathbb{T}$ on $\op{Sig}(\A)$ given by
  \[
    (T\Sigma)_n = H^*_\Sigma n
    \qquad(n\in \A_\fp)
  \]
  is regular. Indeed, for every regular epimorphism $e\colon
  \Sigma\twoheadrightarrow \Gamma$ in $\op{Sig}(\A)$ we have regular
  epimorphisms $e_n\colon \Sigma_n\twoheadrightarrow \Gamma_n$ in $\A$ ($n\in
  \A_\fp$), and the components of $Te$ are the morphisms
  \[
    (Te)_m = \coprod_{n\in |\A_\fp|} \A(n,m) \bigcdot e_n
    \qquad(m\in \A_\fp).
  \]
  Since coproducts of regular epimorphisms in $\A$ are regular epimorphisms, we
  conclude that each $(Te)_m$ is regularly epic in $\A$. Thus, $Te$ is regularly
  epic in $\op{Sig}(\A)$.
  
  Consequently, the category of $\mathbb{T}$-algebras has regular
  factorizations. Since $c$ is a strong epimorphism, it is regular.

  (2)~We can apply Theorem~\ref{T:coequffp}: an algebra $\M$ for $\mathbb{T}$ is
  finitely presentable iff it is a coequalizer in $\Mndf(\A)$ as follows:
  \[
      \xymatrix@1{
      H^*_\Gamma
      \ar@<4pt>[r]^-{\ell}
      \ar@<-1pt>[r]_-{r}
      & H^*_\Sigma
      \ar[r]^-{c}
      & \M
      }
  \]
  for some super-finitary signatures $\Gamma$ and $\Sigma$. By the preceding
  remark, we can substitute $\ell$ and $r$ by a collection of equations
  $\Gamma_n\rightrightarrows H^*_\Sigma n$, and since $\Gamma$ is
  super-finitary, this collection is finite.
  Therefore, every finitely presentable monad in $\Mndf(\A)$ has a
  super-finitary presentation.

  Conversely, let $\M$ be presented by a super-finitary signature $\Sigma$ and
  equations
  \[
    f_i,f_i'\colon k_i\longrightarrow H^*_\Sigma n_i
    \qquad(i=1,\ldots,r).
  \]
  Let $\Gamma$ be the super-finitary signature with
  \[
    \Gamma_k =\coprod_{\substack{i\in I\\ k_i = k}} k_i
    \qquad(\text{for all }k\in \A_\fp)
  \]
  Then we have signature morphisms
  \[
    f,f'\colon \Gamma\longrightarrow T(\Sigma)
  \]
  derived from the given pairs in an obvious way. For the corresponding monad
  morphisms
  \[
    \bar f,\bar f'\colon H^*_\Gamma \longrightarrow H^*_\Sigma
  \]
  we see that the coequalizer of this pair is the smallest quotient $c\colon
  H^*_\Sigma \twoheadrightarrow \M$ with $c_{n_i}\cdot f_i = c_{n_i}\cdot f_i'$
  for all $i=1,\ldots,n$. This follows immediately from the fact that $c$ is a
  regular epimorphism in $\Mndf(\A)$. Indeed, since $\A$ has regular
  factorizations, so does $\op{Sig}(\A)$, a power of $\A$. Moreover,
  $\mathbb{T}$ is a regular monad, thus the category $\Mndf(\A)$ of its algebras has
  regular factorizations, thus, every strong epimorphism is regular.
\end{proof}
\begin{corollary}
  A finitary monad on $\Set$ is a finitely presentable object of $\Mndf(\Set)$
  iff the corresponding variety of algebras has a presentation (in the classical
  sense) by finitely many operations and finitely many equations.
\end{corollary}

Most of ``everyday'' varieties (groups, lattices, boolean algebras, etc.)
yield finitely presentable monads. Vector spaces over a field $K$ yield a finitely presentable monad iff $K$ is finite -- equivalently, that monad is finitely generated. However, there are finitely generated
monads in $\Mndf(\Set)$ that fail to be finitely presentable. 
We  prove that the classes of finitely presentable and  finitely generated  objects differ in $\Mndf(\Set)$ by relating
monads to monoids via an adjunction. 

\begin{remark}
  \begin{enumerate}
  \item 
  Recall that every set functor has a unique strength. This follows
  from the result by Kock~\cite{Kock} that a strength of an
  endofunctor on a closed monoidal category bijectively corresponds to
  a way of making that functor enriched (see also
  Moggi~\cite[Proposition~3.4]{Moggi1991}).
  \item For every monad
  $(T, \eta, \mu)$ on $\Set$ we have a canonical strength, i.e.~a
  family of morphisms
  \[
    s_{X,Y}\colon  TX \times Y \to T(X \times Y)
  \]
  natural in $X$ and $Y$ and such that the following diagrams commute for all
  sets $X,Y,Z$:
  \begin{eqnarray*}
    \begin{tikzcd}[column sep=0mm,row sep=4mm,baseline=(firstrow.base)]
      |[alias=firstrow]|
      TX \times 1
      \arrow{rr}{s_{1,X}}
      & & T(X\times 1)
      \\
      & TX
      \arrow[draw=none]{ul}[sloped,description]{\cong}
      \arrow[draw=none]{ur}[sloped,description]{\cong}
    \end{tikzcd}
    \quad
    \begin{tikzcd}[column sep=5mm,row sep=4mm,baseline=(firstrow.base)]
      |[alias=firstrow]|
      TX \times Y \times Z 
      \arrow{rrd}[below left]{s_{X, Y \times Z}}
      \arrow{rr}{s_{X,Y} \times Z}
      &&
      T(X \times Y) \times Z
      \arrow{d}{s_{X\times Y,Z}}
      \\
      &&
      T(X\times Y\times Z)
    \end{tikzcd}
    \\
    \begin{tikzcd}[column sep=7mm,row sep=4mm,
      every label/.append style={
        inner sep=2mm,
      },baseline=(firstrow.base)
      ]
      |[alias=firstrow]|
      TX \times Y 
      \arrow{r}{s_{X,Y}}
      &
      T(X \times Y)
      \\
      X \times Y
      \arrow{u}{\mathllap{\eta_X \times Y}}
      \arrow{ru}[below right]{\eta_{X\times Y}}
    \end{tikzcd}
    \begin{tikzcd}[column sep=7mm,row sep=4mm,
      every label/.append style={
        inner sep=2mm,
      },
      baseline=(firstrow.base)
      ]
      |[alias=firstrow]|
      TTX \times Y
      \arrow{d}[swap]{\mu_X \times Y}
      \arrow{r}{s_{TX,Y}}
      &
      T(TX \times Y)
      \arrow{r}{Ts_{X,Y}}
      &
      TT(X\times Y)
      \arrow{d}{\mu_{X\times Y}}
      \\
      TX \times Y
      \arrow{rr}{s_{X,Y}}
      &&
      T(X \times Y)
    \end{tikzcd}
  \end{eqnarray*}
  In fact, one defines the canonical
  strength by the commutativity of the following diagrams
  for every element $y\colon  1 \to Y$:
  \[
    \vcenter{
      \xymatrix{
        TX \times Y \ar[rr]^--{s_{X,Y}} && T(X \times Y) \\
        & TX \times 1 \cong TX \cong T(X \times 1)
        \ar `l[lu] [lu]_(.1){TX \times y} \ar `r[ru] [ru]^(.1){T(X \times y)}
      }
    }
  \]
  \end{enumerate}
\end{remark}
\begin{notation}
\begin{enumerate}
\item
We denote by $R\colon  \Mndf(\Set) \to \Mon$ the functor sending a
monad $(T,\eta,\mu)$ with strength $s$ to the monoid $T1$ with unit
$\eta_1\colon  1\to T1$ and the following multiplication:
\[
  m\colon  T1\times T1 \xrightarrow{s_{1,T1}} T(1\times T1) \xrightarrow{\cong} TT1
  \xrightarrow{\mu_1} T1.
\]
For example, the finite power-set monad $\powf$ induces the monoid
$(\{0,1\},\wedge,1)$ of boolean values with conjunction as multiplication.

\item
We define a functor $L\colon \Mon \to \Mndf(\Set)$ as follows. For every monoid $(M, *, 1_M)$ we
have the monad $LM$ of free $M$-sets with the following object assignment,
unit and multiplication:
\[
  LM(X) = M \times X, 
  \quad
  \eta_X \colon  x \mapsto (1_M,x),
  \quad
  \mu_X \colon  (n, (m,x)) \mapsto (n * m, x).
\]
This extends to a functor $L\colon  \Mon \to \Mndf(\Set)$
by $(Lf)_X=f\times \id_X$ for monoid homomorphisms $f$.
\end{enumerate}
\end{notation}

\begin{proposition}
    We have an adjoint situation $L\dashv R$ with the following unit $\nu$ and
    counit $\epsilon$:
    \begin{align*}
      \nu_M\colon& M\xrightarrow{\ \cong\ } M\times 1 = RLM
                   \\
      \epsilon_T\colon & LRT = T1\times (-)
         \xrightarrow{s_{1,-}}
         T(1\times (-)) \xrightarrow{\ T\cong\ } T,
    \end{align*}
    where $s$ is the strength of $T$.
  \end{proposition}
  \begin{proof}
    It is not hard to see that $\nu_M$ is a monoid morphism because the
    monoid structure in $M\times 1 = RLM$ boils down to the monoid structure of
    $M$. Furthermore, $\nu_M$ is clearly natural in $M$. 

    For every monad $T$, $\epsilon_T$
    is a natural  transformation $T1\times (-) \to T$ because of the naturality
    of the strength $s$. The axioms for the strength imply that
    $s_{1,-}\colon  T1\times (-) \to T(1\times (-))$ is a monad morphism by
    straightforward diagram chasing. To see $\nu$ and $\epsilon$ establish an
    adjunction, it remains to check the triangle identities:
    \begin{itemize}
      \item The identity $\epsilon_{LM}\cdot L\nu_M = \id_{LM}$ is just
        the associativity of the product:
        \[
        \begin{tikzcd}[row sep = 1 mm]
          LM
          \arrow{r}{L\nu_M}
          & LRLM \arrow{rr}{\epsilon_{LM}}
          &
          & LM
          \\
          M\times (-)
          \arrow{r}{\nu_M\times (-)}
          & (M\times 1) \times (-)
          \arrow{r}{\cong}
          & M\times (1\times (-))
          \arrow{r}{M\times \cong}
          & M \times (-)
        \end{tikzcd}
        \]
        The composite is obviously the identity on $M\times (-)$.

      \item The identity $R\epsilon_T\cdot \nu_{RT} = \id_{RT}$
        follows directly from the first axiom of  strength:
        \[
        \begin{tikzcd}[row sep = 4 mm, baseline=(T1.base)]
          RT
          \arrow{r}{\nu_{RT}}
          & RLRT \arrow{rr}{R\epsilon_{T}}
          &
          & RT
          \\
          |[alias=T1]|
          T1
          \arrow[shiftarr={yshift={-7mm}}]{rr}{T\nu_1}
          \arrow{r}{\nu_{T1}}
          & T1 \times 1
          \arrow{r}{s_{1,1}}
          & T (1\times 1)
          \arrow{r}{\cong}
          & T 1
        \end{tikzcd}
        \]
      \end{itemize}
      \vspace*{-20pt}
  \end{proof}
\begin{corollary}
  In the category of finitary monads on $\Set$ the classes of finitely
  presentable and finitely generated objects do not coincide.
\end{corollary}
\begin{proof}Note that from the fact that the unit of the adjunction $L \dashv
  R$ is an isomorphism we see that $L$ is fully faithful. Thus, we may regard
  $\Mon$ as a full coreflective subcategory of $\Mndf(\Set)$. Furthermore, the
  right adjoint $R$ preserves filtered colimits; this follows from the fact that
  filtered colimits in $\Mndf(\Set)$ are created by the forgetful functor into
  $[\Set,\Set]_\fin$ where they are formed object-wise. In addition, $L$
  preserves monomorphisms; in fact, for an injective monoid morphism $m\colon  M
  \monoto M'$ the monad morphism $Lm\colon  LM \to LM'$ is monic since all its
  components $m \times \id_X\colon  M \times X \to M' \times X$ are. By
  Lemma~\ref{lem:fppres},
  we therefore have that a monoid $M$ is finitely
  presentable (resp.~finitely generated) if and only if the monad $LM$ is
  finitely presentable (resp.~finitely generated).
  
  In the category $\Mon$ of monoids finitely
  presentable and finitely generated objects do not coincide; see Campbell et
  al.~\cite[Example~4.5]{Campbell96}. Thus the same holds for $\Mndf(\Set)$.
  \end{proof}

\begin{remark}
  For $\lambda$-accessible monads on locally $\lambda$-presentable
  categories all the above results have an appropriate
  statement and completely analogous proofs. We leave the explicit
  formulation to the reader. 
\end{remark}

%
%
\bibliographystyle{myabbrv}
\bibliography{refs}

\begin{thebibliography}{10}

\bibitem{amsw19}
J.~Ad\'amek, S.~Milius, L.~Sousa, and T.~Wi\ss\/mann.
\newblock On finitary functors.
\newblock {S}ubmitted; available online at
  \url{https://arxiv.org/abs/1902.05788}.

\bibitem{AdamekR94}
J.~Ad\'{a}mek and J.~Rosick\'y.
\newblock {\em Locally presentable and accessible categories}.
\newblock Cambridge University Press, 1994.

\bibitem{adamek1974}
J.~Adámek.
\newblock Free algebras and automata realizations in the language of
  categories.
\newblock {\em Comment. Math. Univ. Carolinae}, 015(4):589--602, 1974.

\bibitem{Barr}
M.~Barr.
\newblock Coequalizers and free triples.
\newblock {\em Math. Z.}, 116:307--322, 1970.

\bibitem{Campbell96}
C.~Campbell, E.~Robertson, N.~Ruškuca, and R.~Thomas.
\newblock On subsemigroups of finitely presented semigroups.
\newblock {\em J. Algebra}, 180(1):1 -- 21, 1996.

\bibitem{kellypower93}
G.~Kelly and A.~Power.
\newblock Adjunctions whose counits are coequalizers, and presentations of
  finitary enriched monads.
\newblock {\em J. Pure Appl. Algebra}, 89(1):163 -- 179, 1993.

\bibitem{Kock}
A.~Kock.
\newblock Strong functors and monoidal monads.
\newblock {\em Arch. Math. (Basel)}, 23:113 -- 120, 1972.

\bibitem{lack99}
S.~Lack.
\newblock On the monadicity of finitary monads.
\newblock {\em J. Pure Appl. Algebra}, 140(1):65 -- 73, 1999.

\bibitem{Manes76}
E.~G. Manes.
\newblock {\em Algebraic Theories}.
\newblock Springer-Verlag, 1976.

\bibitem{Moggi1991}
E.~Moggi.
\newblock Notions of computation and monads.
\newblock {\em Inform. and Comput.}, 93(1):55 -- 92, 1991.

\end{thebibliography}

\end{document}